\documentclass{mcom-l}

\newcommand{\R}{\mathbb{R}}
 
\newcommand{\Z}{\mathbb{Z}}
\newcommand{\N}{\mathbb{N}}

\newcommand{\norm}[1]{\left\lVert#1\right\rVert}

\usepackage{amssymb}
\usepackage{graphicx}
\usepackage{xcolor}
\usepackage{booktabs}
\usepackage{tabularx}
\usepackage{adjustbox}
\usepackage{hyperref}
\usepackage{cleveref}
\usepackage{algorithm}
\usepackage{algpseudocode}
\usepackage{float}
\usepackage{caption}
\usepackage{subcaption}
\usepackage{enumerate}
\usepackage{enumitem}   
\usepackage{alltt}
\algtext*{EndWhile}
\algtext*{EndFor}
\algtext*{EndIf}

\copyrightinfo{}{}

\newtheorem{theorem}{Theorem}[section]
\newtheorem{lemma}[theorem]{Lemma}
\newtheorem{corollary}[theorem]{Corollary}

\theoremstyle{definition}
\newtheorem{definition}[theorem]{Definition}
\newtheorem{example}[theorem]{Example}

\newtheorem{question}[theorem]{Open Question}

\theoremstyle{remark}
\newtheorem{remark}[theorem]{Remark}

\numberwithin{equation}{section}

\begin{document}

\title[Deterministic hardness of SVP]{Fine-grained deterministic hardness of\\
	the shortest vector problem}

\author[M. Hittmeir]{Markus Hittmeir}
\address{NORCE Research}
\curraddr{Nygårdsgaten 112, 5008 Bergen, Norway}
\email{mahi@norceresearch.no}
\thanks{The author was supported by the Research Council of Norway (grant 357539).}


\date{}

\dedicatory{}

\begin{abstract}
	Let $\gamma$-$\mathsf{GapSVP}_p$ be the decision version of the shortest vector problem in the $\ell_p$-norm with approximation factor $\gamma$, let $n$ be the lattice rank and $0<\varepsilon\leq 1$. We prove that there is no algorithm that solves $(2-\varepsilon)$-$\mathsf{GapSVP}_p$ uniformly for all $p\in\N$ in time 
	\[
	2^{2^{o(p)}}\cdot 2^{o(n)},
	\] 
	unless the Exponential Time Hypothesis is false. The proof is based on a deterministic Karp reduction from a constrained variant of the subset-sum problem to $\mathsf{GapSVP}_p$ for fixed $p$. While most hardness results for the shortest vector problem in finite norms rely on randomized reductions, our method is entirely deterministic. As a consequence, we also obtain a deterministic Karp reduction from the standard subset-sum problem to $(2-\varepsilon)$-$\mathsf{GapSVP}_{\infty}$.
\end{abstract}

\maketitle

\section{Introduction}
The Shortest Vector Problem (SVP) is a fundamental problem in computational mathematics, lattice theory and cryptography. The goal is to find the shortest non-zero vector (with respect to some $\ell_p$-norm) that can be written as an integer linear combination of certain vectors over $\R^m$ (the set of these linear combinations is commonly referred to as a \emph{lattice}). SVP is the foundation of several public-key cryptosystems (\cite{Pei}) and is known to be NP-hard under randomized reductions (\cite{Ajt}, \cite{Mic}). However, proving its hardness under deterministic reductions has been a notorious open problem (\cite{BenPei}). We contribute to this problem by proving results on the \emph{fine-grained} deterministic hardness of SVP, which provide lower bounds for its runtime that are conditional on core hypotheses in complexity theory. 

Let us briefly recall the formal definition of lattices and some fundamental notions related to them.
\begin{definition}
	Let $n,m\in\N$. A \emph{lattice} $\mathcal{L}$ of rank $n$ and dimension $m$ is the set of all integer linear combinations of linearly independent vectors $\textbf{b}_1,\ldots,\textbf{b}_n\in\R^m$,
	\[
	\mathcal{L}=\mathcal{L}(\textbf{b}_1,\ldots,\textbf{b}_n):=\left\{\sum_{i=1}^n y_i \textbf{b}_i: y_1,\ldots, y_n\in\Z\right\}.
	\]
	The \emph{$\ell_p$-norm} of $\textbf{x}=(x_1,\ldots,x_m)\in \R^m$ is defined as
	$
	||\textbf{x}||_p:=\left(\sum_{i=1}^m|x_i|^p\right)^{1/p}
	$
	for $1\leq p <\infty$, and as $||\textbf{x}||_{\infty}:=\max_{1\leq i \leq m} |x_i|$ for $p=\infty$.
	
	We define $\lambda_{1,p}(\mathcal{L})$ to be the \emph{length of a shortest non-zero vector} in $\mathcal{L}$ in the $\ell_p$-norm, 
	\[
	\lambda_{1,p}(\mathcal{L}):=\min_{\textbf{v}\in\mathcal{L}\backslash \{0\}}||\textbf{v}||_p.
	\]
\end{definition}

\begin{definition}[$\gamma$-$\mathsf{GapSVP}$ in the $\ell_p$-norm]	
	\leavevmode\par\noindent
	Let $\gamma\geq 1$ and $p\geq 1$. We define the $\gamma$-approximative \emph{decision version of the shortest vector problem} in the $\ell_p$-norm, $\gamma$-$\mathsf{GapSVP}_p$, as follows: Given a lattice $\mathcal{L}$ of rank $n$ and a distance threshold $\delta>0$, decide whether
	\begin{itemize}
		\item{$\lambda_{1,p}(\mathcal{L})\leq \delta$ (YES-instance)},
		\item{$\lambda_{1,p}(\mathcal{L})> \gamma\delta$ (NO-instance)},	
	\end{itemize}
when one of these cases is promised to hold.
For $\gamma=1$, we denote the \emph{exact} decision version of the shortest vector problem by $\mathsf{GapSVP}_p$.
\end{definition}

Hardness results for $\gamma$-$\mathsf{GapSVP}_p$ have been studied extensively. The NP-hardness of $\mathsf{GapSVP}_\infty$ was first established in 1981 by van Emde Boas (\cite{Van}). Dinur (\cite{Din}) proved in 2002 that $\gamma$-$\mathsf{GapSVP}_\infty$ with the almost polynomial approximation factor $\gamma=n^{c/\log\log n}$ for some constant $c>0$ is NP-hard. Hardness results for finite $p$ started with Ajtai's work in 1998 (\cite{Ajt}), who proved that $\mathsf{GapSVP}_2$ is NP-hard under a randomized reduction. Subsequently, the hardness has been extended to constant approximation factors for all finite $p$ (\cite{Kho}, \cite{Mic}). While these results are still based on randomized reductions (assuming $\text{NP}\not\subseteq \text{RP}$), Micciancio (\cite{Mic}) also gave a deterministic polynomial-time reduction under the conjecture that, for every $\varepsilon>0$, there exists $d$ such that, for all sufficiently large $n$, $[n,n+n^{\varepsilon}]$ contains a square-free integer whose prime factors are all at most $\log^d n$. Most hardness results for SVP are based on a reduction from the closest vector problem, using locally dense lattices. However, recent work by Hair and Sahai (\cite{HaiSah}) establishes unconditional deterministic NP-hardness of the shortest vector problem for $p>2$ with the approximation factor $2^{\log^{1-\varepsilon}(n)}$. Their reduction bypasses the closest vector problem and uses probabilistically checkable proofs (PCP) instead, resolving the long-standing open problem of proving deterministic NP-hardness for any finite $p$. In addition, it improves upon the constant approximation factors of the randomized reductions discussed above. For a more detailed overview on results and open problems concerning SVP, we refer the reader to Bennett's survey (\cite{Ben}).

Another interesting research area is the \emph{fine-grained} hardness of SVP. The main goal is to provide a lower bound for the runtime complexity of SVP based on an assumption such as the Strong Exponential Time Hypothesis ($\mathsf{SETH}$). This hypothesis states that for every $\varepsilon>0$ there is a $k\in\N$ such that the $k$-SAT problem on formulas with $n$ variables is not in $O(2^{(1-\varepsilon)n})$-time\footnote{All cost statements in this paper treat arithmetic and bitwise operations on $O(L)$-bit integers as unit-cost, where $L$ is the maximum bit-length of any integer appearing in the instance encoding.}. Similar to the NP-hardness results for finite $p$ discussed above, most known results providing such lower bounds for $\gamma$-$\mathsf{GapSVP}_p$ for finite $p$ are based on the randomized version of this hypothesis. In 2018, it has been proved (\cite{Agg}) that for $p>p_0\approx 2.1397$, $p\not\in 2\Z$, $\mathsf{GapSVP}_p$ is not in $O(2^{C_p n})$-time for some constant $C_p\in(0,1)$, unless randomized $\mathsf{SETH}$ is false. By using a gap-variant of $\mathsf{SETH}$ that states a complexity lower bound for the approximation of the number of satisfiable clauses in $k$-SAT, this has also been extended to  $\gamma$-$\mathsf{GapSVP}_p$ for some constant $\gamma>1$ (\cite{Agg},\cite{BenPeiTan}). In addition, Hecht and Safra (\cite{HecSaf}) recently proved two hardness results for $\mathsf{GapSVP}_p$ based on reductions that are subexponential-time, but deterministic. The first states that $(\sqrt{2}-o(1))$-$\mathsf{GapSVP}_p$ is not in polynomial-time for every $p>2$, unless 3-SAT is in $O(2^{O(n^{2/3}\log n)})$-time. The second statement improves upon the approximation factor $\gamma$ for sufficiently large $p$, assuming $\text{NP}\not\subseteq \text{SUBEXP}$.

To summarize, there is a large variety of strong hardness results for $\mathsf{GapSVP}_p$. But, with few recent exceptions, the results rely on randomized reductions. In this paper, we will contribute to the problem of providing stronger hardness results for GapSVP that do not rely on randomization. In particular, we are interested in a \emph{parameterized} version of the shortest vector problem. Such versions have previously been studied with the distance threshold $\delta$ as the parameter (\cite{BenCheGur}, \cite{BhaBonEgr}). We will instead consider SVP parameterized by the norm parameter $p$. To be precise, we investigate the following problem, where $p$ is a part of the input.

\begin{definition}[$\gamma$-$\mathsf{GapSVP}_p$ parameterized by $p$]
	\leavevmode\par\noindent
	Let $\gamma\geq 1$. We define the $\gamma$-approximative \emph{parameterized decision version of the shortest vector problem} in the $\ell_p$-norm, $\gamma$-$\mathsf{GapSVP}_{\mathrm{par}}$, as follows:
	Given a lattice $\mathcal{L}$ of rank $n$, a distance threshold $\delta>0$, and the parameter $p\in\mathbb{N}$, decide whether
	\begin{itemize}
		\item $\lambda_{1,p}(\mathcal{L})\leq \delta$ \quad (YES-instance),
		\item $\lambda_{1,p}(\mathcal{L})> \gamma\delta$ \quad (NO-instance),
	\end{itemize}
	when one of these cases is promised to hold.
\end{definition}

We start by establishing a \emph{deterministic polynomial-time Karp reduction} from a variant of the subset-sum problem to standard $\mathsf{GapSVP}_p$ for all $p\geq 1$. The reduction preserves the problem dimension $n$ up to constant factors, which allows to make use of a fine-grained hardness result for subset-sum based on the Exponential Time Hypothesis ($\mathsf{ETH}$). This hypothesis is weaker than $\mathsf{SETH}$ and states that 3-SAT cannot be solved in time $2^{o(n)}$. Here is our main result.

\begin{theorem}\label{thm:main}
		Let $0<\varepsilon\leq 1$. Assuming $\mathsf{ETH}$, there is no algorithm that solves $(2-\varepsilon)$-$\mathsf{GapSVP}_{\mathrm{par}}$ in time $2^{2^{o(p)}}\cdot 2^{o(n)}$.
\end{theorem}

In addition, our techniques yield a deterministic Karp reduction from the standard subset-sum problem to the constant-factor approximate shortest vector problem in the $\ell_\infty$-norm, with only linear blow-up in the lattice rank.  A similar direct reduction to the exact problem $\mathsf{SVP}_\infty$ was claimed in \cite{MicGol}. However, it has been demonstrated in \cite{KreNip} that the argument does not hold. In addition to van Emde Boas' well-known reduction from bounded homogenous linear equations, the following result thus provides a novel route to studying the hardness of $\gamma$-$\mathsf{GapSVP}_{\infty}$.

\begin{theorem}\label{thm:main2}
	Let $0<\varepsilon\leq 1$. There is a deterministic polynomial-time reduction from the subset-sum problem to $(2-\varepsilon)$-$\mathsf{GapSVP}_{\infty}$ that reduces instances with $n$ weights to instances with a lattice of rank $O_{\varepsilon}(n)$.
\end{theorem}

As mentioned, our proofs are based on a Karp reduction from a variant of the subset-sum problem to GapSVP. The standard version of subset-sum considers $a_1x_1+\cdots +a_nx_n=s$, where the $a_i$ (the ``weights'') and $s$ (the ``target sum'') are given non-negative integers. The problem is to find $(x_1,\ldots,x_n)\in\{0,1\}^{n}$ such that the equation is satisfied. Our variant adds restrictions for integer vectors that are orthogonal to both the vector of weights $(a_1,\ldots,a_n)\in\R^n$ and to the vector $\textbf{1}_n=(1,1,\ldots,1)\in \R^n$. Namely, if no such non-zero vector that is short in the $\ell_p$-norm exists, then the problem instance is reducible to an instance of GapSVP. We subsequently show that any standard subset-sum problem instance satisfies these restrictions with respect to some $\ell_p$-norm for a sufficiently small value of $p$, allowing us to prove the bound stated in \Cref{thm:main}. In addition, if one were to prove NP-hardness of said variant of subset-sum, this would immediately imply NP-hardness of $\mathsf{GapSVP}_p$ for fixed $p\geq 1$ (in particular for $p=2$).

\vspace{12pt}
\textbf{Comparison with concurrent deterministic hardness results for SVP.} 
The first preprint versions of the results of Hair and Sahai (\cite{HaiSah}), Hecht and Safra (\cite{HecSaf}), and the present paper all appeared within a few weeks, in October and November 2025. While all three works establish deterministic reductions to variants of SVP, their results are complementary: they use different source problems and reduction frameworks, address different versions of the shortest vector problem, and study different approximation and parameter regimes. 

\begin{table}[htbp] \centering \caption{Recent deterministic hardness results for SVP.} \label{tab:comp} \footnotesize \renewcommand{\arraystretch}{1.15} \setlength{\tabcolsep}{3pt} \begin{tabularx}{\textwidth}{ l *{6}{>{\raggedright\arraybackslash}X} } \toprule \textbf{Paper} & \textbf{Reduction} & \textbf{Framework} & \textbf{Problem} & \textbf{Parameter $p$} & \textbf{Approx.\ factor} & \textbf{Hardness type} \\ 
		\midrule Hair/Sahai & Polynomial & PCP & $\mathsf{GapSVP}_{p}$ & Fixed $p>2$ & $2^{\log^{1-\varepsilon} n}$ & NP-hard \\ 
		\addlinespace Hecht/Safra & Sub-exponential & PCP/CSP & $\mathsf{GapSVP}_{p}$, $\mathsf{uSVP}_{p}$ & Fixed $p>2$ & $\sqrt{2}-o(1)$ & Fine-grained \\ 
		\addlinespace This paper & Polynomial & Subset-Sum & $\mathsf{GapSVP}_{\mathrm{par}}$, $\mathsf{GapSVP}_{\infty}$ & $p\in\N$ as input & $2-\varepsilon$ & Fine-grained \\ \bottomrule \end{tabularx} 
\end{table}

Hair and Sahai prove deterministic NP-hardness of $\mathsf{GapSVP}_{p}$ for every fixed $p>2$. Hecht and Safra provide fine-grained hardness results for every fixed $p>2$ using subexponential-time reductions from 3-SAT and Unambiguous-3-SAT. They study both $\mathsf{GapSVP}_{p}$ and $\mathsf{uSVP}_{p}$, the latter denoting instances promised to have a unique shortest vector. In this paper, we consider $p$ as part of the input and study a parameterized version of the shortest vector problem, using subset-sum as the source problem. Our hardness result is also fine-grained in the sense that we assume  $\mathsf{ETH}$ and its resulting lower bound for the complexity of subset-sum, but our reduction is polynomial-time.

\textbf{Note on recent claims concerning $\mathsf{ETH}$.} While this paper was under review, two preprints (\cite{Cla}, \cite{Ope}) claiming the refutation of the Exponential Time Hypothesis were released. Both preprints are accompanied by formalizations in Lean, but independent assessment of the results is ongoing. If these claims turn out to be correct, then $\mathsf{ETH}$ would be false, and the $\mathsf{ETH}$-based conclusion of Theorem \ref{thm:main} would not rule out an algorithm for $\mathsf{GapSVP}_{\mathrm{par}}$ with the stated running time. Our main deterministic polynomial-time reductions (Corollary \ref{cor:reduction} and Theorem \ref{thm:main2}), as well as all other results that do not depend on $\mathsf{ETH}$, remain unaffected. In particular, any known or conjectured complexity lower bound for subset-sum can still be transferred to $\mathsf{GapSVP}_{\mathrm{par}}$ and $\mathsf{GapSVP}_{\infty}$ via our reductions.

\textbf{Overview.} The remainder of the paper is structured as follows: In \Cref{sec:rel_work}, we will state lemmata and discuss related work, in particular the technique for solving subset-sum problems of low density. In \Cref{sec:improv}, we will explain our method and prove the main results. \Cref{sec:discuss} provides remarks and directions for future research.

\vspace{12pt}
\textbf{Acknowledgment.} The author would like to thank two anonymous referees for discovering an issue in the proof of Lemma 3.8 in an earlier version of this paper, and for comments that helped to substantially improve the overall presentation.

\section{Preliminaries} \label{sec:rel_work}
Our main strategy is to exploit and elaborate on a relationship between the shortest vector and the subset-sum problem. We will discuss the core principle of this relationship later in this section. Let us start with some standard and/or easily obtainable results. The following lemma is well-known.
\begin{lemma}\label{lem:holder}
	Let $1\leq q\leq p$, $m\in\N$ and $\textbf{x}\in\R^m$. Then
$
	||\textbf{x}||_q\leq m^{1/q-1/p}||\textbf{x}||_p.
$
Moreover, for $1\leq q < p=\infty$, we have $||\textbf{x}||_{\infty}\leq||\textbf{x}||_q\leq m^{1/q}||\textbf{x}||_{\infty}$.
\end{lemma}
\begin{proof} 
	Let $\textbf{x}=(x_1,\ldots,x_m)$. The case $p=q$ is immediate. For $q < p < \infty$, we apply H\"older's inequality for finite sums
	\[ 
	\sum_{i=1}^m |a_i b_i| \leq \left(\sum_{i=1}^m |a_i|^r\right)^{1/r} \left(\sum_{i=1}^m |b_i|^s\right)^{1/s}, \qquad \frac1r+\frac1s=1,
	\]
	with $a_i=|x_i|^q$, $b_i=1$, $r=p/q$ and $s=p/(p-q)$. This gives $ \|\mathbf{x}\|_q^q \leq \|\mathbf{x}\|_p^q m^{1-q/p}$, and taking $q$-th roots proves the claim.
	For $p=\infty$, the first inequality follows from $|x_i|\leq\|\mathbf{x}\|_q$ and the second from $|x_i|\leq\|\mathbf{x}\|_\infty$ for every $i\leq m$.
 \end{proof}

\begin{lemma}\label{lem:infinity}
	Let $1\leq \gamma < \gamma'$, let $\mathcal{L}$ be a lattice of dimension $m$, let $\delta>0$ and $q\in\N$ such that
	\[
	q\geq \log_2 m / \log_2\left(\gamma'/\gamma\right).
	\]	
	$\gamma\text{-}\mathsf{GapSVP}_{\infty}(\mathcal{L},\delta)$ is a YES-instance if $\gamma'\text{-}\mathsf{GapSVP}_{\mathrm{par}}(\mathcal{L},\delta,q)$ is a YES-instance.
	Similarly, $\gamma\text{-}\mathsf{GapSVP}_{\infty}(\mathcal{L},\delta)$ is a NO-instance if $\gamma'\text{-}\mathsf{GapSVP}_{\mathrm{par}}(\mathcal{L},\delta,q)$ is a NO-instance. 	
\end{lemma}
\begin{proof}
	Since the second inequality in \Cref{lem:holder} holds pointwise for all $\textbf{x}\in\mathcal{L}$, it also holds for the minima over all $\textbf{x}\in\mathcal{L}\backslash \{0\}$, hence
	\[
	\lambda_{1,\infty}(\mathcal{L})\leq \lambda_{1,q}(\mathcal{L})\leq m^{1/q}\lambda_{1,\infty}(\mathcal{L}).
	\]
	Now if $\gamma'$-$\mathsf{GapSVP}_{\mathrm{par}}(\mathcal{L},\delta,q)$ is a YES-instance, then 
	$
	\lambda_{1,\infty}(\mathcal{L}) \leq \lambda_{1,q}(\mathcal{L})\leq \delta,
	$
	hence $\gamma$-$\mathsf{GapSVP}_{\infty}(\mathcal{L},\delta)$ is also a YES-instance. If $\gamma'$-$\mathsf{GapSVP}_{\mathrm{par}}(\mathcal{L},\delta,q)$ is a NO-instance, then 
	$
	m^{1/q}\lambda_{1,\infty}(\mathcal{L})\geq\lambda_{1,q}(\mathcal{L})>\gamma'\delta,
	$
	hence 
	\[
	\lambda_{1,\infty}(\mathcal{L})>\gamma'\delta/ m^{1/q}\geq \gamma'\delta\cdot \left(\gamma/\gamma'\right)=\gamma\delta.
	\]
	Therefore, $\gamma$-$\mathsf{GapSVP}_{\infty}(\mathcal{L},\delta)$ is also a NO-instance.
\end{proof}

As mentioned in the introduction, we will consider a variant of the subset-sum problem that imposes restrictions on the weights. In this context, we now introduce a notion that we will call $p$-dependency. 
\begin{definition}\label{def:rel}
	Let $0<\alpha<1$ and $p\geq 1$. We call non-negative integers $a_1,\ldots,a_n$ \emph{$p$-dependent with respect to $\alpha$} if there exists $\textbf{x}=(x_1,\ldots,x_n)\in\Z^n$ with $\textbf{x}\neq \textbf{0}$ such that
	\[
	\sum_{i=1}^n a_i x_i=0 \,\wedge\, \sum_{i=1}^n x_i=0 \,\wedge\, ||\textbf{x}||_p\leq \alpha\cdot n^{1/p}.
	\]
	If the value of $\alpha$ is clear from the context, we simply call $a_1,\ldots,a_n$ \emph{$p$-dependent}.
\end{definition}

\begin{example}
	Let $n=25$, $a_1=1$, $a_2=5$, $a_3=9$ and $a_4,\ldots,a_n\in\N_0$. For $\textbf{x}=(-1,2,-1,0,\ldots,0)\in\Z^n$, we have $||\textbf{x}||_2=\sqrt{6}\leq 5/2$, and the other conditions hold too. Hence, $a_1,\ldots,a_n$ are $2$-dependent with respect to $\alpha=1/2$.
\end{example}

\begin{example}\label{exa:rel}
	Let $p\geq 1$, $n\in\N$, $0<\alpha<1$ and $N:=\lfloor \alpha\cdot n^{1/p}\rfloor+1$. Then $N,N^2,\ldots,N^n$ are not $p$-dependent with respect to $\alpha$. Namely, if we assume that there is $\textbf{x}=(x_1,\ldots,x_n)\in\Z^n$ that satisfies the conditions of Definition \ref{def:rel}, then
	\[
	0=\sum_{i=1}^n N^{i}x_i=N\sum_{i=1}^nN^{i-1}x_i=N\left(x_1+N\sum_{i=2}^nN^{i-2}x_i\right)
	\]
	implies $x_1\equiv 0 \mod N$. From $|x_1|\leq ||\textbf{x}||_p \leq \alpha\cdot n^{1/p} < N$ it follows that $x_1=0$. Repeating this argument inductively, we get $x_i=0$ for all $i\leq n$, a contradiction.
\end{example}

\begin{lemma}\label{lem:rel}
	Let $0<\alpha<1$, $p\geq 1$, $n,m\in\N$ and $a_1,\ldots,a_n, b_1,\ldots,b_m$ be non-negative integers. The following hold:
	\begin{enumerate}
		\item{$a_1,\ldots,a_n$ are $p$-dependent iff $\lambda a_1,\ldots, \lambda a_n$ are $p$-dependent for every $\lambda\in\N$}.
		\item{If $a_1,\ldots,a_n$ are $p$-dependent, then $a_1,\ldots,a_n,b_1,\ldots,b_m$ are $p$-dependent.}
		\item{If $a_1,\ldots,a_n$ are $p$-dependent, they are also $q$-dependent for every $q\in[1,p]$.}
		\item{If $n<2\alpha^{-p}$, then $a_1,\ldots,a_n$ are not $p$-dependent.}
	\end{enumerate}
\end{lemma}

\begin{proof}
	One observes $(1)$ by noting that, for any $\textbf{x}\in\Z^n$, we have $\sum_{i=1}^n a_i x_i=0$ if and only if $\sum_{i=1}^n (\lambda a_i) x_i=0$ for every $\lambda\in\N$. For $(2)$, let $\textbf{x}=(x_1,\ldots,x_n)$ be a vector that satisfies the conditions of \Cref{def:rel} for $a_1,\ldots,a_n$. Then $\textbf{x}'=(x_1,\ldots,x_n,0,\ldots,0)$ in $\Z^{n+m}$ satisfies all conditions of \Cref{def:rel} for $a_1,\ldots a_n,b_1,\ldots, b_m$, in particular 
	\[
	||\textbf{x}'||_p=||\textbf{x}||_p\leq \alpha\cdot n^{1/p}\leq \alpha\cdot (n+m)^{1/p}.
	\]
	
	For proving $(3)$, assume that there exists $\textbf{x}=(x_1,\ldots,x_n)\in\Z^n$ satisfying the condition of \Cref{def:rel}.
	Considering the third condition, \Cref{lem:holder} implies
	$
	||\textbf{x}||_{q}\leq n^{1/q-1/p}||\textbf{x}||_p\leq \alpha\cdot n^{1/q},
	$
	which proves the claim.
	
	Let us now prove $(4)$. Assume that $a_1,\ldots,a_n$ are $p$-dependent. We first note that $\textbf{x}=(x_1,\ldots,x_n)$ must have at least two non-zero entries. The smallest possible combination (leading to the smallest $\ell_p$-norm) would be $1$ and $-1$, which implies $n\geq \alpha^{-p}||\textbf{x}||_p^p=2\alpha^{-p}$. Clearly, for all other combinations of two or more non-zero entries, $n$ must be even larger. We have thus established that, if $a_1,\ldots,a_n$ are $p$-dependent, then $n\geq 2\alpha^{-p}$. The statement $(4)$ follows.
\end{proof}

\begin{remark} While we will use the notion of $p$-dependency exactly as it is stated above, it may be interesting to note the following:
	\begin{itemize}
		\item[(i)]{	Define $\textbf{a}=(a_1,\ldots,a_n)\in \Z^n$ and $\textbf{1}_n=(1,1,\ldots,1)\in\Z^n$, then the vector $\textbf{x}$ in \Cref{def:rel} is orthogonal to both $\textbf{a}$ and $\textbf{1}_n$, i.e.,
			\[
			\textbf{a} \cdot \textbf{x} = 0 \,\text{ and }\, \textbf{1}_n \cdot \textbf{x} = 0.
			\]
		}
		\item[(ii)]{Imposing certain restrictions on $a_1,\ldots,a_n$ allows to increase the bound in \Cref{lem:rel} (4). For example, if the $a_1,\ldots,a_n$ are distinct, one can prove the same statement for $n<4\alpha^{-p}$.}
	\end{itemize}
\end{remark}

Let us move on to explaining an approach for solving subset-sum problems via an oracle for the shortest vector problem. It is well known that the LLL-algorithm (\cite{LenLenLov}) for lattice reduction can be used to solve a certain class of knapsack and subset-sum problems in polynomial-time (\cite{LagOdl}, \cite{SchEuc}). In particular, consider a subset-sum instance with weights $a_i$ and target sum $s$, given by the equation
\begin{align}\label{eq:ssum}
	a_1x_1+\cdots +a_nx_n=s.
\end{align}
The problem is to find $(x_1,\ldots,x_n)\in\{0,1\}^{n}$ such that $(\ref{eq:ssum})$ is satisfied. In the context of lattice reduction techniques, the efficient solvability of this problem depends on the density 
$
d:=n/\log_2(\max_i a_i).
$
The best of these results has been proved by Coster et al. (\cite{CosJouMac}), where it is shown that an oracle for finding the shortest vector in a special lattice can be used to solve almost all subset-sum problems with $d<0.9408$. In practice, this oracle is replaced by an application of the LLL-algorithm or other lattice reduction techniques. The procedure works as follows. For $N:=\lfloor \frac{1}{2}\sqrt{n}\rfloor+1$, we define the lattice $\mathcal{L}\subseteq \Z^{n+1}$ spanned by
\begin{align*}
	\textbf{b}_1 &= (1,0,\ldots,0,0,Na_1),\\
	\textbf{b}_2 &= (0,1,\ldots,0,0,Na_2),\\
	&\vdots\\
	\textbf{b}_{n}&=(0,0,\ldots,0,1,Na_n),\\
	\textbf{b}_{n+1}&=\left(\frac{1}{2},\frac{1}{2}\ldots,\frac{1}{2},\frac{1}{2},Ns\right).
\end{align*}
If $(x_1,\ldots,x_n)$ is a solution to (\ref{eq:ssum}), we have
\[
\textbf{v}=\sum_{i=1}^n x_i\textbf{b}_i-\textbf{b}_{n+1}=(y_1,\ldots,y_n,0)\in\mathcal{L},
\]
where $y_i\in\{-\frac{1}{2},\frac{1}{2}\}$ and, hence, $\norm{\textbf{v}}_2\leq \frac{1}{2}\sqrt{n}$. Knowing $\textbf{v}$, it is easy to retrieve a solution to the subset-sum problem. Therefore, the hope of applying the LLL-algorithm is that $v$ will occur in the reduced basis of $\mathcal{L}$.

In our approach (\Cref{lem:lattice}), we will consider a variation of the lattice defined above that also takes the number of non-zero elements in $(x_1,\ldots, x_n)$ into account. Using this restriction together with $p$-independency of the weights, we can rigorously prove that $\textbf{v}$ is the shortest vector in the lattice without any assumption regarding the density of the subset-sum problem. 

\section{Main Results}\label{sec:improv}
In this section, we will prove \Cref{thm:main} and \Cref{thm:main2}. We start by formally defining the decision version of the subset-sum problem, and consider some hardness results from the literature.
\begin{definition}\label{def:ssum}
	We define the $\emph{subset-sum problem}$, $\mathsf{SS}$, as follows:
	Given $s\in\N$ and non-negative integers $a_1,\ldots,a_n$, decide whether
	\begin{itemize}
		\item{$\exists (x_1,\ldots,x_n)\in\{0,1\}^{n}: \sum_{i=1}^n a_ix_i=s$ (YES-instance)}
		\item{$\nexists (x_1,\ldots,x_n)\in\{0,1\}^{n}: \sum_{i=1}^n a_ix_i=s$ (NO-instance)}		
	\end{itemize}
	For $\mathsf{SS}(a_1,\ldots,a_n,s)$, we will call witnesses $(x_1,\ldots,x_n)$ for YES-instances \emph{solutions} of $\mathsf{SS}(a_1,\ldots,a_n,s)$.
\end{definition}

\begin{lemma}\label{lem:ssnp}
	$\mathsf{SS}$ is \emph{NP}-hard.
\end{lemma}

\begin{proof}	
Subset-sum is the eighteenth problem in Karp's list of NP-complete problems from 1972 (see \cite[p. 95]{Kar}, under ``Knapsack''). The proof is by reduction from the partition problem. The only difference is that Karp's original definition assumes all weights to be positive, whereas we also allow the weights to be $0$. Since every instance with positive weights is also an instance with non-negative weights, this generalization preserves NP-hardness.
\end{proof}

As mentioned in the introduction, there is a sub-exponential lower bound for the runtime of subset-sum based on the Exponential Time Hypothesis. The following statement is well known (e.g., see \cite{JanLanLan}).
\begin{lemma}\label{lem:eth}
	Assuming $\mathsf{ETH}$, $\mathsf{SS}$ cannot be solved in time $2^{o(n)}$.
\end{lemma}

In addition, a recent improvement also established SETH-hardness of subset-sum in terms of a lower bound depending on the weight/target size. 
\begin{lemma}\label{lem:seth}
	Assuming $\mathsf{SETH}$, for any $\varepsilon>0$ there exists $\delta>0$ such that  $\mathsf{SS}$ is not in time $O(s^{1-\varepsilon}2^{\delta n})$.
\end{lemma}

\begin{proof}
	This is Theorem 1.1 in \cite{Abb}. The proof establishes a reduction from $k$-SAT to subset-sum instances that provides a much tighter lower bound than previous reductions, such as those from standard NP-hardness proofs.
\end{proof}

Consider subset-sum problem instances $\mathsf{SS}(a_1,\ldots,a_n,s)$ as defined in \Cref{def:ssum}. We start by showing that we can impose a restriction on the number of weights without losing hardness.

\begin{definition} \label{def:2ss}
	We define the \emph{$1/2$-full subset-sum problem}, $\mathsf{\mathsf{SS}_{1/2}}$, as follows:
	Given $s\in\N$ and non-negative integers $a_1,\ldots,a_n$ with $n$ being even, decide whether
	\begin{itemize}
		\item{$\exists (x_1,\ldots,x_n)\in\{0,1\}^{n}: \sum_{i=1}^n a_ix_i=s \wedge |\{i:x_i=1\}|=\frac{n}{2}$ (YES-instance)}
		\item{$\nexists (x_1,\ldots,x_n)\in\{0,1\}^{n}: \sum_{i=1}^n a_ix_i=s \wedge |\{i:x_i=1\}|=\frac{n}{2}$ (NO-instance)}		
	\end{itemize}
\end{definition}

\begin{lemma}\label{lem:2ss}
	There is a deterministic polynomial-time reduction from  $\mathsf{SS}$ to $\mathsf{SS}_{1/2}$ that reduces instances with $n$ weights to instances with $2n$ weights.
\end{lemma}

\begin{proof}
	Let $\mathsf{SS}(a_1,\ldots,a_n,s)$ be any subset-sum problem instance. 
	Let 
	\[
	\mathsf{SS}_{1/2}(a_1,\ldots,a_n,0,\ldots,0,s),
	\] 
	be an instance of $\mathsf{SS}_{1/2}$ where $n$ weights of zeros are added in addition to the $n$ weights $a_1,\ldots,a_n$ of the original problem. Now if $\mathsf{SS}_{1/2}(a_1,\ldots,a_n,0,\ldots,0,s)$ has a solution, the fact that $s$ is positive implies that the solution must include some of the weights $a_i$, so $\mathsf{SS}(a_1,\ldots,a_n,s)$ also has a solution. On the other hand, any solution of $\mathsf{SS}(a_1,\ldots,a_n,s)$ containing $k\leq n$ weights can be turned into a solution of $\mathsf{SS}_{1/2}(a_1,\ldots,a_n,0,\ldots,0,s)$ by adding $n-k$ zeros. We conclude that $\mathsf{SS}(a_1,\ldots,a_n,s)$ has a solution if and only if $\mathsf{SS}_{1/2}(a_1,\ldots,a_n,0,\ldots,0,s)$ has a solution.
\end{proof}

\begin{definition} \label{def:npss}
	Let $c\in(0,1)$ be a constant. We define the \emph{c-full subset-sum problem}, $\mathsf{\mathsf{SS}}_{c}$, as follows:
	Given $s\in\N$ and non-negative integers $a_1,\ldots,a_n$ such that $cn\in\N$, decide whether
	\begin{itemize}
		\item{$\exists (x_1,\ldots,x_n)\in\{0,1\}^{n}: \sum_{i=1}^n a_ix_i=s \wedge |\{i:x_i=1\}|=cn$ (YES-instance)}
		\item{$\nexists (x_1,\ldots,x_n)\in\{0,1\}^{n}: \sum_{i=1}^n a_ix_i=s \wedge |\{i:x_i=1\}|=cn$ (NO-instance)}		
	\end{itemize}
By $\mathsf{SS}_{c}^{p,\alpha}$, we denote the c-full subset-sum problem that is restricted to inputs $a_1,\ldots,a_n$ that are \emph{not} $p$-dependent with respect to $\alpha$ (see \Cref{def:rel}).
\end{definition}

\begin{lemma}\label{lem:npss}
	Let $m\in\N$ be arbitrary but fixed. For both $c=(m+1/2)/(m+1)$ and $c=1/(2(m+1))$, there is a deterministic polynomial-time reduction from  $\mathsf{SS}_{1/2}$ to $\mathsf{\mathsf{SS}}_{c}$ that reduces instances with $n$ weights to instances with $(m+1)n$ weights.
\end{lemma}

\begin{proof}
	Let $\mathsf{SS}_{1/2}(a_1,\ldots,a_n,s)$ be any $\mathsf{SS}_{1/2}$ instance. We start by considering the case $c:=(m+1/2)/(m+1)$. Let $b_1,\ldots,b_{mn}$ be natural numbers, all equal to $B:=1+\max\{s,\sum_{i=1}^n a_i\}$. Denote $\beta:=\sum_{j=1}^{mn} b_j=mnB$ and consider the target sum $s':=s+\beta$ together with the $mn+n$ integers
	$
	a_1,\ldots,a_n,b_1,\ldots,b_{mn}.
	$ 
	Note that
	$
	c(mn+n)=mn+n/2\in\N.
	$
	Since
	\[
	\sum_{i=1}^n a_i + (mn-1)B < \beta < s',
	\] 
	any solution of $\mathsf{\mathsf{SS}}_{c}(a_1,\ldots,a_n,b_1,\ldots,b_{mn},s')$ must contain all $mn$ weights $b_j$ and, consequently, exactly $n/2$ weights among the $a_i$. 
	Since all the $b_j$ sum to $\beta$, the remaining $n/2$ weights among the $a_i$ sum to $s$. Therefore, they are also a solution of $\mathsf{SS}_{1/2}(a_1,\ldots,a_n,s)$.  On the other hand, any solution of $\mathsf{SS}_{1/2}(a_1,\ldots,a_n,s)$ can be turned into a solution of $\mathsf{\mathsf{SS}}_{c}(a_1,\ldots,a_n,b_1,\ldots,b_{mn},s')$ by simply adding all the weights $b_j$. We observe that $\mathsf{SS}_{1/2}(a_1,\ldots,a_n,s)$ has a solution if and only if $\mathsf{\mathsf{SS}}_{c}(a_1,\ldots,a_n,b_1,\ldots,b_{mn},s')$ has a solution.
	
	Let us now prove the case $c:=1/(2(m+1))$. We consider the same $mn+n$ integers
	$
	a_1,\ldots,a_n,b_1,\ldots,b_{mn},
	$
	as above, but change the target to $s':=s$. Note that $c(mn+n)=n/2\in\N$. Now any solution of $\mathsf{\mathsf{SS}}_{c}(a_1,\ldots,a_n,b_1,\ldots,b_{mn},s')$ must omit the $b_j=B>s'$ and contain exactly $n/2$ weights among the $a_i$.  Again we obtain that $\mathsf{\mathsf{SS}}_{c}(a_1,\ldots,a_n,b_1,\ldots,b_{mn},s')$ has a solution if and only if $\mathsf{SS}_{1/2}(a_1,\ldots,a_n,s)$ has a solution. 
\end{proof}

\begin{lemma}\label{lem:sschard}
	Let $m\in\N$ be arbitrary but fixed. For both $c=(m+1/2)/(m+1)$ and $c=1/(2(m+1))$, the following holds.
	\begin{enumerate}
		\item{$\mathsf{\mathsf{SS}}_{c}$ is \emph{NP}-hard.}
		\item{Assuming $\mathsf{ETH}$, $\mathsf{SS}_{c}$ cannot be solved in time $2^{o(n)}$.}
	\end{enumerate}
\end{lemma}

\begin{proof}
	The first statement follows directly from the reductions in \Cref{lem:2ss} and \Cref{lem:npss} and from \Cref{lem:ssnp}. 
	
	For proving the second statement, assume to the contrary that $\mathsf{SS}_{c}$ can be solved in time $2^{o(n)}$. Let $\mathsf{SS}(a_1,\ldots,a_l,s)$ be any subset-sum problem instance. Considering \Cref{lem:2ss} and \Cref{lem:npss}, one observes that this $\mathsf{SS}$ instance can be reduced to a $\mathsf{SS}_c$ instance with $n=Ml$ weights, where $M:=2(m+1)$ is a fixed constant. Hence, our assumption allows us to solve this $\mathsf{SS}_c$ instance and, thus, our initial $\mathsf{SS}$ instance, in time 
	$
	2^{o(Ml)}=2^{o(l)}.
	$
Since the instance was arbitrary, \Cref{lem:eth} implies that $\mathsf{ETH}$ is false, which we wanted to show.	
\end{proof}

\begin{remark}
	The case $c=1/(2(m+1))$ in Lemma \ref{lem:npss} and Lemma \ref{lem:sschard} has been proved for symmetry and completeness, but it is not used in the remainder of this paper.
\end{remark}

The most crucial step towards our main results is the following lemma, which concerns a reduction from $\mathsf{SS}_{c}^{p,\alpha}$ instances to GapSVP instances. We exploit both the $p$-independency of the weights of $\mathsf{SS}_{c}^{p,\alpha}$ and the restriction on the number of weights in its solutions. These two properties can be used to prove that the shortest vector in a lattice similar to the one discussed at the end of \Cref{sec:rel_work} must always correspond to a solution of $\mathsf{SS}_{c}^{p,\alpha}$. In the proofs of the main results later in this section, we will then show that $\mathsf{SS}_{c}^{p,\alpha}$ instances are just $\mathsf{SS}_{c}$ instances for certain families of inputs $p,\alpha$ and $n$, which allows us to apply the hardness results of \Cref{lem:sschard}.

\begin{lemma}\label{lem:lattice}
	Let $p\geq 1$, $0<\varepsilon\leq 1$ and $m\in\N$ such that
	\begin{equation}\label{eq:assumptions}
	c:=\frac{m+1/2}{m+1}>1-\frac{\varepsilon}{4} \hspace{12pt}\text{ and }\hspace{12pt} \alpha:=1-\frac{\varepsilon}{2}.
	\end{equation}
	Let $\mathsf{SS}_{c}^{p,\alpha}(a_1,\ldots, a_n, s)$ be an instance of $\mathsf{SS}_{c}^{p,\alpha}$. Define $N:=\lfloor \alpha\cdot n^{1/p}\rfloor + 1$ and $r:=cn\in\N$ and consider the lattice $\mathcal{L}_r\subseteq \Z^{n+2}$ spanned by the vectors
	\begin{align*}
		\textbf{b}_1 &= (1,0,\ldots,0,0,N,Na_1),\\
		\textbf{b}_2 &= (0,1,\ldots,0,0,N,Na_2),\\
		&\vdots\\
		\textbf{b}_{n}&=(0,0,\ldots,0,1,N,Na_n),\\
		\textbf{b}_{n+1}&=\left(\frac{1}{2},\frac{1}{2},\ldots,\frac{1}{2},Nr,Ns\right).
	\end{align*}
	If either $\varepsilon=1$ or $n<(3^p-1)/((2-\varepsilon)^p-1)$ for $\varepsilon<1$, the following hold:
	\begin{enumerate}
		\item{If $(x_1,\ldots,x_n)$ solves $\mathsf{SS}_{c}^{p,\alpha}(a_1,\ldots, a_n, s)$, then
			$
			\left(x_1-\frac{1}{2},\ldots,x_n-\frac{1}{2},0,0\right) 
			$
			is a shortest vector of $\mathcal{L}_r$, and it holds that $\lambda_{1,p}(\mathcal{L}_r)=n^{1/p}/2$.}
		\item{If $||\textbf{v}||_p\leq (2-\varepsilon)\cdot n^{1/p}/2$ for some vector $\textbf{v}=(v_1,\ldots,v_{n+2})\in\mathcal{L}_r$, then either $(v_1+\frac{1}{2},v_2+\frac{1}{2},\ldots, v_{n}+\frac{1}{2})$ or $(-v_1-\frac{1}{2},-v_2-\frac{1}{2},\ldots, -v_{n}-\frac{1}{2})$ is a solution of $\mathsf{SS}_{c}^{p,\alpha}(a_1,\ldots, a_n, s)$.}
	\end{enumerate}
\end{lemma}

\begin{proof}
	By definition, any solution of $\mathsf{SS}_{c}^{p,\alpha}(a_1,\ldots, a_n, s)$  has exactly $r$ non-zero entries in the corresponding vector $(x_1,\ldots,x_n)$. One easily observes that $\mathcal{L}_r$ contains the vector
	$
	\textbf{s}=(s_1,\ldots,s_{n},0,0), 
	$
	where $s_{i}=x_i-1/2$. Note that
	$
	||\textbf{s}||_p^p =n/2^p. 
	$ 
	Let $\mathcal{S}$ be the set that, for all solutions of $\mathsf{SS}_{c}^{p,\alpha}(a_1,\ldots, a_n, s)$, contains $\textbf{s}$ and -$\textbf{s}$. Our goal is to show that $\mathcal{S}$ is the complete set of $(2-\varepsilon)$-approximate shortest vectors in $\mathcal{L}_r$. Consider all vectors $\hat{\textbf{x}}=(\hat{x}_1,\ldots,\hat{x}_{n+2})$ which satisfy the three conditions
	\begin{align*}
		||\hat{\textbf{x}}||_p&\leq (2-\varepsilon)\cdot n^{1/p}/2=\alpha\cdot n^{1/p},\\
		\hat{\textbf{x}}&\in \mathcal{L}_r,\\
		\hat{\textbf{x}}&\not\in\mathcal{S}\cup\{\textbf{0}\}.
	\end{align*}
	In order to prove both claims (1) and (2), it is enough to show that no such vector exists. Assume to the contrary that it does. Then $N>\alpha\cdot n^{1/p}$ implies $\hat{x}_{n+1}=0$ and $\hat{x}_{n+2}=0$.
	Writing $\hat{\textbf{x}}=\sum_{i=1}^{n} y_i \textbf{b}_i + y \textbf{b}_{n+1}$ establishes the following equations.
	\begin{align}
		\hat{x}_i = & y_i + y/2 \text{ for $1\leq i\leq n$},\label{e1}\\
		0 = & \hat{x}_{n+1} =N\left( y r + \sum_{i=1}^{n} y_{i}\right),\label{e2}\\
		0 = & \hat{x}_{n+2} = N\left(\sum_{i=1}^{n} a_{i} y_{i} + ys\right).\label{e4}
	\end{align}
	We now prove that $|y|<2$. We start by showing that, for all $\hat{\textbf{x}}$ satisfying our conditions,
	\begin{align}
		\left|\sum_{i=1}^{n} y_{i}\right|\leq \frac{n(|y|+2-\varepsilon)}{2}.\label{e7}
	\end{align}
	As a consequence of (\ref{e1}), we have $|y_{i}|=|\hat{x}_{i}-y/2|\leq |\hat{x}_{i}|+|y/2|$, from which it follows that
	\[
	\left|\sum_{i=1}^{n} y_{i}\right|\leq \frac{n|y|}{2} + \sum_{i=1}^n|\hat{x}_{i}|.
	\]
	Since $||\hat{\textbf{x}}||_p\leq \alpha\cdot n^{1/p}$ and $\hat{x}_{n+1}=\hat{x}_{n+2}=0$, applying \Cref{lem:holder} to the first $n$ coordinates of $\hat{\mathbf{x}}$ implies that
	\begin{align*}
		\sum_{i=1}^n|\hat{x}_{i}|=||\hat{\textbf{x}}||_1\leq n^{1-1/p}||\hat{\textbf{x}}||_p
		\leq \alpha\cdot n = (2-\varepsilon)\cdot n/2.
	\end{align*}
	We conclude that (\ref{e7}) holds. Together with (\ref{e2}), we obtain
	\[
	r=\frac{1}{|y|}\left|\sum_{i=1}^{n} y_{i}\right|\leq \frac{n}{2}\left(1+\frac{2-\varepsilon}{|y|}\right).
	\]
	For $|y|\geq 2$, using (\ref{eq:assumptions}) in this inequality yields
	$
	r\leq (1-\varepsilon/4)\cdot n < cn = r,
	$
	a contradiction. We conclude that $|y|$ has to be smaller than $2$.

	We are hence left with the cases $y=0,1,-1$. Let $y=-1$ and note that $|\hat{x}_{i}|=|y_{i}-1/2|$ due to (\ref{e1}). Assume first that
	\begin{align}
		y_{i}=0 \,\,\,\vee\,\,\, y_{i}=1\label{e9}
	\end{align}
	holds for all $i\leq n$. We either have $\hat{x}_{i}=0-1/2=-1/2$ or $\hat{x}_{i}=1-1/2=1/2$. In any case, $|\hat{x}_{i}|^p=1/2^p$ for all $i$, and $||\hat{\textbf{x}}||_p=n^{1/p}/2$. However, $(\ref{e4})$ yields a solution to $\mathsf{SS}_{c}^{p,\alpha}$. Hence, the third condition $\hat{\textbf{x}}\not\in\mathcal{S}\cup\{\textbf{0}\}$ is not satisfied. 
	
	Assume now that there is at least one $i$ for which $(\ref{e9})$ is not true. It follows that $|\hat{x}_{i}|=|y_{i}-1/2|\geq 3/2$ and, hence,
	\[
	||\hat{\textbf{x}}||_p\geq ((n-1)/2^p+(3/2)^p)^{1/p}.
	\]
	We apply our assumptions to prove that $((n-1)/2^p+(3/2)^p)^{1/p}>(2-\varepsilon)\cdot n^{1/p}/2$. If $\varepsilon=1$, we have $(n-1)/2^p+(3/2)^p=(n-1+3^p)/2^p>n/2^p$, and the claim follows by taking $p$-th roots. We may hence focus on $\varepsilon<1$, where we assumed that $n<(3^p-1)/((2-\varepsilon)^p-1)$. This is equivalent to
	\begin{align*}
	&\hspace{28pt} 3^p > (2-\varepsilon)^p\cdot n-(n-1)\\
	&\Leftrightarrow\hspace{12pt} (3/2)^p +(n-1)/2^p > (2-\varepsilon)^p\cdot n/2^p,\\
	&\Leftrightarrow\hspace{12pt} ((n-1)/2^p+(3/2)^p)^{1/p}>(2-\varepsilon)\cdot n^{1/p}/2.
	\end{align*}
	It follows that $||\hat{\textbf{x}}||_p > (2-\varepsilon)\cdot n^{1/p}/2=\alpha\cdot n^{1/p}$. Hence, for $y=-1$, no $\hat{\textbf{x}}$ exists that satisfies the conditions. With very similar arguments, the same can be proved for $y=1$.  
	
	Let us now deal with the final remaining case, $y=0$. From $(\ref{e2})$ we may deduce that $\sum_{i=1}^ny_{i}=0$. Define $\textbf{y}=(y_1,y_2,...,y_{n})\in \Z^n$, then (\ref{e1}) implies that
	\[
	||\textbf{y}||_p=||\hat{\textbf{x}}||_p\leq (2-\varepsilon)\cdot n^{1/p}/2=\alpha\cdot n^{1/p}.
	\]
	Since $\hat{\textbf{x}}\neq\textbf{0}$ we also have $\textbf{y}\neq\textbf{0}$. Together with $(\ref{e4})$, this entails that $a_1,\ldots,a_n$ are $p$-dependent with respect to $\alpha$, a contradiction. So for $y=0$, there is no $\hat{\textbf{x}}$ satisfying the conditions. This finishes the proof.
\end{proof}

\begin{corollary}\label{cor:reduction}
	Let $p\geq 1$, $c=5/6$ and $\alpha=1/2$. There is a deterministic polynomial-time reduction from $\mathsf{SS}_{c}^{p,\alpha}$ to $\mathsf{GapSVP}_p$ that reduces instances with $n$ weights to instances with a lattice of rank $n+1$.
\end{corollary}

\begin{proof}
	Let $\mathsf{SS}_{c}^{p,\alpha}(a_1,\ldots, a_n, s)$ be any instance of $\mathsf{SS}_{c}^{p,\alpha}$, and apply \Cref{lem:lattice} with the given value for $p$, with $\varepsilon=1$ and $m=2$. Note that the inequality in (\ref{eq:assumptions}) is satisfied. We compute $N$ and, for $r=5n/6$, construct the lattice $\mathcal{L}_r$ spanned by the $\textbf{b}_i$ in negligible time. Then the statements (1) and (2) guarantee that, for the distance threshold $\delta:=n^{1/p}/2$, $\mathsf{SS}_{c}^{p,\alpha}(a_1,\ldots, a_n, s)$ is a YES-instance if and only if $\mathsf{GapSVP}_p(\mathcal{L}_r,\delta)$ is a YES-instance.
\end{proof}

Based on Lemma \ref{lem:lattice}, we may now prove our two main results for $\gamma$-$\mathsf{GapSVP}_{\mathrm{par}}$ and $\gamma$-$\mathsf{GapSVP}_{\infty}$ with approximation factors of the form $\gamma=2-\varepsilon$, as stated in the introduction.

\begin{proof}[Proof of \Cref{thm:main}]
	Let $0<\varepsilon\leq 1$ be arbitrary. Set $\alpha:=1-\varepsilon/2$ and, for a sufficiently large $m\in\N$, set $c:=(m+1/2)/(m+1)>1-\varepsilon/4$.
	Assume to the contrary that there exists an algorithm that solves $(2-\varepsilon)$-$\mathsf{GapSVP}_{\mathrm{par}}$ in time  $2^{2^{o(p)}}\cdot 2^{o(n)}$. We show that this implies that $\mathsf{ETH}$ is false.
	
	Let $\mathsf{SS}_{c}(a_1,\ldots,a_n,s)$ be any instance of $\mathsf{SS}_c$. For $p >(\log_2 n - 1)/\log_2(1/\alpha)$, \Cref{lem:rel} (4) implies that
	$
	\mathsf{SS}_{c}(a_1,\ldots,a_n,s)=\mathsf{SS}_{c}^{p,\alpha}(a_1,\ldots,a_n,s).
	$
	We compute the slightly larger value $\hat{p}:=\lceil\log_2 (n +1)/\log_2(1/\alpha)\rceil$, and apply \Cref{lem:lattice} with our values for  $\hat{p}$, $\varepsilon$ and $m$. We compute $N$ and, for $r=cn$, construct the lattice $\mathcal{L}_r$ spanned by the $\textbf{b}_i$ with negligible cost. We have 
	\[
	\hat{p}\geq \log_2 (n +1)/\log_2(2/(2-\varepsilon))>\log_2 (n +1)/\log_2(3/(2-\varepsilon)).
	\]
	Hence, for $\varepsilon<1$, it follows that
	\[
	n< \frac{3^{\hat{p}}}{(2-\varepsilon)^{\hat{p}}} -1 <\frac{3^{\hat{p}}}{(2-\varepsilon)^{\hat{p}}} - \frac{1}{(2-\varepsilon)^{\hat{p}}}<\frac{3^{\hat{p}}-1}{(2-\varepsilon)^{\hat{p}}-1}.
	\]
	Consequently, all requirements for \Cref{lem:lattice} are satisfied. The statements (1) and (2) guarantee that, for the distance threshold $\delta:=n^{1/\hat{p}}/2$, $\mathsf{SS}_{c}^{\hat{p},\alpha}(a_1,\ldots, a_n, s)$ is a YES-instance if and only if 
	\[
	(2-\varepsilon)\text{-}\mathsf{GapSVP}_{\hat{p}}(\mathcal{L}_r,\delta)= (2-\varepsilon)\text{-}\mathsf{GapSVP}_{\mathrm{par}}(\mathcal{L}_r,\delta,\hat{p})
	\]
	 is a YES-instance. It follows from our assumption that we can solve this instance in
	\[
	2^{2^{o(\hat{p})}}\cdot 2^{o(n)} =	2^{2^{o(\log n)}}\cdot 2^{o(n)} = 2^{o(n)}
	\] 
	arithmetic operations. As a consequence of \Cref{lem:lattice}, the same is true for the considered $\mathsf{SS}_{c}$ instance.
	Since the instance was arbitrary, we obtain an algorithm for $\mathsf{SS}_{c}$ that runs in time $2^{o(n)}$. 
	Using \Cref{lem:sschard} (2), it follows that $\mathsf{ETH}$ is violated, which we wanted to show.
\end{proof}

\begin{proof}[Proof of \Cref{thm:main2}]
	Let $0<\varepsilon\leq 1$ be arbitrary. Set $\varepsilon':=\varepsilon/2$, $\alpha:=(2-\varepsilon)/(2-\varepsilon')$ and, for a sufficiently large $m\in\N$, set $c:=(m+1/2)/(m+1)>1-\varepsilon'/4$. We will show that $\mathsf{SS}_{c}$ reduces to $(2-\varepsilon)$-$\mathsf{GapSVP}_{\infty}$, and the stated result easily follows via \Cref{lem:2ss} and \Cref{lem:npss}.
	
	Let $\mathsf{SS}_{c}(a_1,\ldots,a_n,s)$ be any instance of $\mathsf{SS}_c$. We follow the proof of \Cref{thm:main} above. Since \Cref{lem:rel} (4) implies 
	$
	\mathsf{SS}_{c}(a_1,\ldots,a_n,s)=\mathsf{SS}_{c}^{p,\alpha}(a_1,\ldots,a_n,s)
	$
	for $p >(\log_2 n - 1)/\log_2(1/\alpha)$,
	we apply \Cref{lem:lattice} with 	$
	\hat{p}:=\lceil\log_2 (n + 2)/\log_2(1/\alpha)\rceil
	$, $\varepsilon'$ and $m$. One observes that 
	\[
	\hat{p}\geq \log_2 (n + 2)/\log_2((2-\varepsilon')/(2-\varepsilon))>\log_2 (n +1)/\log_2(3/(2-\varepsilon')).
	\]
	Hence, all requirements for \Cref{lem:lattice} are satisfied. It follows that, for the distance threshold $\delta:=n^{1/\hat{p}}/2$, $\mathsf{SS}_{c}^{\hat{p},\alpha}(a_1,\ldots, a_n, s)$ is a YES-instance if and only if 
	$(2-\varepsilon')\text{-}\mathsf{GapSVP}_{\mathrm{par}}(\mathcal{L}_r,\delta,\hat{p})$
	is a YES-instance. Since $\mathcal{L}_r$ is a lattice of dimension $n+2$, we observe that $\hat{p}$ is sufficiently large to apply \Cref{lem:infinity} with $\gamma'=2-\varepsilon'$ and $\gamma=2-\varepsilon$. As a consequence, we can decide $(2-\varepsilon')\text{-}\mathsf{GapSVP}_{\mathrm{par}}(\mathcal{L}_r,\delta,\hat{p})$ by deciding $(2-\varepsilon)\text{-}\mathsf{GapSVP}_{\infty}(\mathcal{L}_r,\delta)$. Retracing the arguments of the proof, we see that the same is true for $\mathsf{SS}_{c}(a_1,\ldots,a_n,s)$. Since this instance was arbitrary, the claim is shown.
\end{proof}

\section{Discussion}\label{sec:discuss}

Let us conclude by discussing some properties of the reduction proved in this paper, as well as possibilities to improve and extend the results, and other directions to explore in future research.

\begin{remark}(Scalability of the reduction)
	
	For the sake of convenience, we chose $\hat{p}$ slightly larger than required in the proof of \Cref{thm:main}. Considering $\varepsilon=1$ and $c=5/6$ like in \Cref{cor:reduction}, one observes that any instance $\mathsf{SS}_{5/6}(a_1,\ldots,a_n,s)$ can be reduced to an instance of $\mathsf{GapSVP}_{p}$ for $p> \log_2 n-1$. Following back the reductions (\Cref{lem:2ss} and \Cref{lem:npss}) to an original $\mathsf{SS}(a'_1,\ldots,a'_m,s')$ instance, we see that $\mathsf{SS}_{5/6}(a_1,\ldots,a_n,s)$ has six times as many weights as the original instance, i.e., $n=6m$. It follows that any subset-sum instance with $m$ weights can be reduced to an instance of $\mathsf{GapSVP}_{\hat{p}}$ for $\hat{p}:=\lceil \log_2 (6m)-1\rceil$, regardless of the density of the problem. For example, all subset-sum problem instances with $m=10000$ weights (or less) can be reduced to instances of $\mathsf{GapSVP}_{15}$, and instances with $m=10^{82}$ weights (or less) can be reduced to instances of $\mathsf{GapSVP}_{274}$. 
\end{remark}

\begin{remark}(Size of the next shortest vectors)
	
	For even $y$ (in particular, for the case $y=0$), the next shortest vector must contain at least two entries of $1$ and $-1$ in the components $\hat{x}_{i}$, leading to a total length of at least $2^{1/p}$. So for increasing $p$, we have $\lim_{p\to\infty} n^{1/p}/2=1/2$ for the shortest vector, but $\lim_{p\to\infty} 2^{1/p}=1$ for the next shortest vector. For odd $y$ (in particular for $y=\pm 1$), the next shortest vector must contain at least one entry $\pm 3/2$ in addition to the other entries of shape $\pm 1/2$, so for increasing $p$, we have  
	\[
	\lim_{p\to\infty} ((n-1)/2^p+(3/2)^p)^{1/p}=3/2.
	\]
	We conjecture that the approximation factor $\gamma=2-\varepsilon$ is optimal for our technique, at least in its currently applied form.	
	
	The idea of using $1/2$ in the components of $\textbf{b}_{n+1}$ has first been introduced in Coster et al.'s improvement  (\cite{CosJouMac}) for solving low-density subset-sum instances. As it increases the gap between the shortest vector and the next shortest vectors, it is also a crucial element of our approach. The resulting factor of $1/2$ in the length of the shortest vector allows us to deduce an upper bound for $n$ in \Cref{lem:rel} (4) that grows with $p$ and can be used to prove our main results. However, since we are working with subset-sum instances $\mathsf{SS}_{c}$ with $c$ close to $1$, it is natural to consider using $\textbf{b}_{n+1}=(c,c,\ldots,c,Nr,Ns)$ instead, which leads to a shortest lattice vector of size
	\[
	n^{1/p} (c(1-c)^p+c^p(1-c))^{1/p}.
	\]
	For fixed $p$ and $c\rightarrow 1$, this is indeed shorter than $n^{1/p}/2$ (although not short enough to prove NP-hardness of $\mathsf{GapSVP}_{p}$ directly). For growing $p$ and fixed $c<1$, however, this advantage vanishes. Unfortunately, it thus appears like our main results cannot be improved by using this replacement of $\textbf{b}_{n+1}$.
\end{remark}

\begin{remark}(SETH-hardness of $\mathsf{GapSVP}_{\mathrm{par}}$)
	
We have proved and formulated \Cref{thm:main} under $\mathsf{ETH}$ using \Cref{lem:eth}. However, it is also possible to use \Cref{lem:seth} instead to obtain a basis-sensitive refinement of \Cref{thm:main}. 
 Namely, consider $(2-\varepsilon)$-$\mathsf{GapSVP}_{\mathrm{par}}$ with its three inputs $\mathcal{L}=\mathcal{L}(\mathcal{B})$, $\delta$ and $p$, where the lattice $\mathcal{L}$ is given by a specific basis $\mathcal{B}$. Let $L=||\mathcal{B}||_\infty$ be the largest absolute value of any integer component in $\mathcal{B}$. Then for every $\alpha>0$ there is $\beta>0$ such that $(2-\varepsilon)$-$\mathsf{GapSVP}_{\mathrm{par}}$ cannot be solved in time
\[
2^{2^{o(p)}}\cdot L^{1-\alpha}\cdot 2^{\beta n},
\]
unless $\mathsf{SETH}$ fails. An analogous result transfers to $(2-\varepsilon)$-$\mathsf{GapSVP}_{\infty}$, with the bound $L^{1-\alpha}\cdot 2^{\beta n}$. This formulation is somewhat unconventional, since the parameter $L$ depends on the chosen basis representation and is not lattice invariant. However, it is natural for the subset-sum lattices produced by our reduction, where $L$ directly reflects the size of the weights of the source instance.
\end{remark}

\begin{remark}(Reduction from $\mathsf{GapSVP}_{\infty}$ to $\mathsf{GapSVP}_{\mathrm{par}}$)

Lemma 6.6 in \cite{BenGolStep} implies that $2$-$\mathsf{GapSVP}_{\infty}$ is ETH-hard, namely by applying the reduction from $k$-SAT with $k=3$ and approximation factor $1+2/(k-1)=2$. Analogous to \Cref{lem:infinity}, one can show that it is possible to decide $\gamma'$-$\mathsf{GapSVP}_{\infty}(\mathcal{L},\delta')$ by deciding $\gamma$-$\mathsf{GapSVP}_{\mathrm{par}}(\mathcal{L},\delta,q)$ for sufficiently large $q$ and $\gamma<\gamma'$. 
This suggests that one could derive an ETH-hardness result for $(2-\varepsilon)$-$\mathsf{GapSVP}_{\mathrm{par}}$ similar to the one proved in this paper, but based on a reduction from $\mathsf{GapSVP}_{\infty}$ instead of subset-sum. Interestingly, both routes lead to the same approximation factor.
\end{remark}

\begin{remark}(Fixed-parameter tractability of $\mathsf{GapSVP}_{\mathrm{par}}$)
	
Consider the \emph{fixed-parameter tractable} ($\mathsf{FPT}$) complexity class (\cite{FluGro}) consisting of parameterized problems that can be decided in running time
$
f(k)\cdot |x|^{O(1)}
$
where $|x|$ is the input size, $k\in\N$ is the parameter and $f$ is an arbitrary (computable) function only depending on $k$. Depending on the growth of $f$, we can also define the classes for (sub-)exponential fixed-parameter tractability $\mathsf{SUBEPT}$ and $\mathsf{EPT}$, and note that $\mathsf{SUBEPT}\subseteq \mathsf{EPT} \subseteq \mathsf{FPT}$. Our reduction implies that $(2-\varepsilon)$-$\mathsf{GapSVP}_{\mathrm{par}}$ is not in $\mathsf{EPT}$, i.e., cannot be solved in time $2^{O(p)}\cdot |x|^{O(1)}$, unless $\text{P}=\text{NP}$. However, the same non-membership in $\mathsf{EPT}$ follows for larger approximation factors via the reduction from $\mathsf{GapSVP}_{\infty}$ to $\mathsf{GapSVP}_{\mathrm{par}}$ discussed in the remark above, together with the NP-hardness of approximation results known for $\mathsf{GapSVP}_{\infty}$. 
\end{remark}

\begin{remark}(Relevance to cryptography)
	
	As mentioned in the introduction, the presumed hardness of lattice problems such as SVP is widely used as a foundation for the security of cryptosystems, particularly those designed to withstand attacks by quantum algorithms. The results of this paper strengthen the theoretical evidence for the hardness of SVP by establishing deterministic reductions from the subset-sum problem and proving conditional complexity lower bounds. However, they do not directly provide security guarantees or concrete security estimates for existing lattice-based cryptosystems, since the hardness results underlying them typically concern the $\ell_2$-norm and approximation factors that grow polynomially with the lattice dimension.
\end{remark}

\begin{remark}(NP-hardness of $\mathsf{GapSVP}_{p}$ for fixed $p$)

Besides improving the construction of the lattice, \Cref{cor:reduction} gives rise to another direction for future research. It implies that we could provide a proof of NP-hardness of $\mathsf{GapSVP}_{p}$ by showing NP-hardness of $\mathsf{SS}_{5/6}^{p,1/2}$. We plan on investigating the potential of this approach and will study the notion of $p$-dependency in this context.
\end{remark}

\begin{question}
	Is $\mathsf{SS}_{5/6}^{p,1/2}$ NP-hard for any $p\geq 1$?
\end{question}


\end{document}